\documentclass{amsart}
\usepackage{amsmath, amsthm, amscd, amsfonts, amssymb, graphicx}
\usepackage[bookmarksnumbered, plainpages]{hyperref}

\usepackage{mathrsfs}
\usepackage{tikz-cd}

\newtheorem{thm}{Theorem}[section]
\newtheorem{cor}[thm]{Corollary}

\newtheorem{prop}[thm]{Proposition}
\newtheorem{defn}[thm]{Definition}

\numberwithin{equation}{section}

\def\tr{\mbox{tr}}

\def\dim{\mbox{dim}}
\def\Vol{\mbox{Vol}}
\def\Area{\mbox{Area}}

\begin{document}

\title{\bf Remarks on Weakly $p$-K\"ahler hyperbolic manifolds}
\author{Changpeng Pan}

\address{Changpeng Pan\\School of Mathematics and Statistics\\
Nanjing University of Science and Technology\\
Nanjing, 210094,P.R. China\\ }
\email{mathpcp@njust.edu.cn}


\maketitle

\begin{abstract}
This paper investigates the generalizations and applications of weakly $p$-K\"ahler hyperbolic manifolds. 
\end{abstract}

\vskip 0.2 true cm


\pagestyle{myheadings}
\markboth{\rightline {\scriptsize C. Pan}}
         {\leftline{\scriptsize Remarks on weakly $p$-K\"ahler hyperbolic manifolds}}

\bigskip
\bigskip


\section{ Introduction}
As is well known, the Hopf conjecture states that a $2n$-dimensional compact Riemannian manifold with negative sectional curvature should have positive sign Euler characteristic number, i.e. $(-1)^{n}\chi(X)>0$. For the case of $n=1,2$, it can be derived from the Chern-Gauss-Bonnet formula. For general dimensions, Gromov \cite{G} introduced the concept of K\"ahler hyperbolic manifolds, and prove that it has positive sign Euler characteristic number. Specifically, a K\"ahler manifold with negative sectional curvature is a K\"ahler hyperbolic manifold, so he proved the Hopf conjecture for K\"ahler manifold. He also proved that K\"ahler hyperbolicity implies Kobayashi hyperbolicity and the canonical bundle of K\"ahler hyperbolic manifold is quasi-ample.

Since then, people have recognized the importance of K\"ahler hyperbolicity and numerous works on the generalizations and applications of K\"ahler hyperbolicity have been published. Here, we only cite a portion of them, such as \cite{BDET,CX,CY,FWW,HM,H,JZ,L,MP}.

Recently, Bei-Diverio-Eyssidieux-Trapani \cite{BDET}  introduced the concept of weakly K\"ahler hyperbolic manifolds in their paper and explored their relationship with Kobayashi hyperbolicity and the distribution of entire curves. The fundamental property of weakly K\"ahler hyperbolicity is its birational invariance. Through appropriate suitable modifications, they demonstrated that there exists a spectral gap for the Laplacian on weakly K\"ahler hyperbolic manifolds, thereby proving the conjectures of Lang and Green–Griffiths. 

A K\"ahler manifold is called Kobayashi hyperbolic if there is no non-constant holomorphic map $f:\mathbb{C}\rightarrow X$.  Gromov proved that K\"ahler hyperbolicity implies Kobayashi hyperbolicity. Chen-Yang \cite{CY} generalized this result to symplectic manifolds that are homotopy equivalent to Riemannian manifolds with negative sectional curvature. Inspired by their work, Li \cite{L} showed that there are no rational curves on K\"ahler exact manifold, i.e., no non-constant holomorphic map $f:\mathbb{CP}^{1}\rightarrow X$, which implies that the canonical bundle of a K\"ahler exact manifold of general type is ample. 

Marouani-Popovici \cite{MP} and Haggui-Marouani \cite{HM} considered the existence of non-constant holomorphic maps $f:\mathbb{C}^{p}\rightarrow X$, thereby introducing the concepts of balanced hyperbolicity and $p$-K\"ahler hyperbolicity. They studied the properties of these concepts.

Inspired by the above work, this paper aims to investigate the distribution of non-constant rational curves on weakly $p$-K\"ahler hyperbolic (exact) manifolds.

Let $(X,\omega)$ be a compact K\"ahler manifold.  Let $\pi:\widetilde{X}\rightarrow X$ be the K\"ahler universal covering map. A closed $p$-form $\alpha$ on $X$ is called $\widetilde{d}(\text{bounded})$ if the pull back $p$-form $\widetilde{\alpha}$ on $\widetilde{X}$ is of the form $\widetilde{d}\beta$, and $\|\beta\|_{L^{\infty}(\widetilde{X},\widetilde{\omega})}<+\infty$. 
 Cohomology calss $[\mu]\in H^{1,1}(X,\mathbb{R})$ is called big and nef if $[\mu]\geq 0$ and $\int_{X}[\mu]^n>0$. 
\begin{defn}\cite{BDET}
$(X,\omega)$ is called weakly K\"ahler hyperbolic if there exists a closed real form $\mu$ whose cohomology class $[\mu]\in H^{1,1}(X,\mathbb{R})$ is big and nef, such that it is $\widetilde{d}(\text{bounded})$. In particular, if $[\mu]$ is K\"ahler class it is called K\"ahler hyperbolic. 
\end{defn}

\medskip

In \cite{BDET}, the authors proved that weakly K\"ahler hyperbolic manifolds possess some interesting analytic and topological properties, For example, they showed that the universal cover of a weakly K\"ahler hyperbolic manifold satisfies the linear isoperimetric inequality, its subvarieties outside the degeneracy set have infinite fundamental groups, and the image of all entire curves must lie within the degeneracy set, among others. The presented paper aims to generalize this result to weakly $p$-K\"ahler manifolds.

In \cite{H}, Huang considered holomorphic bundles on K\"ahler hyperbolic manifolds, extending the results of Gromov \cite{G} and Li \cite{L} to certain holomorphic bundles. We also consider holomorphic bundles on weakly K\"ahler hyperbolic manifolds.

\medskip

This paper is organized as follows. In Section \ref{sec:Pre}, we give the difinition and examples of weakly $p$-K\"ahler hyperbolic manifolds. In Section \ref{sec:weakly}, we prove some analytic and topological properties of weakly $p$-K\"ahler hyperbolic manifolds. In Section \ref{sec:vb}, we prove that for holomorphic vector bundles $E$ over weakly K\"ahler hyperbolic manifolds $X$ satisfies certain condition then $\chi^{n}(X,E)>0$.

\section{Basic definitions and properties}\label{sec:Pre}
Motivated by \cite{DP}, we give the following definition.
\begin{defn}
Let $X$ be a compact K\"ahler manifold. A $(p,p)$-cohomology $[\alpha]\in H^{p,p}(X,\mathbb{R})$ is called nef if there exists a K\"ahler metric $\omega$ for any 
$m$-dimensional irreducible analytic subvarieties $V\subset X$ it holds that $\int_{V}\alpha\wedge\omega^{m-p}\geq 0$, where $m=p,\cdots,n$. Its null locus is defined by
\begin{equation*}
Null([\alpha])=\cup_{\int_{Z}\alpha\wedge\omega^{m-p}=0}Z
\end{equation*}
where $Z$ varies among all proper analytic subvarieties of $X$ with $m\geq p$. We call a nef $(p,p)$-cohomology is strictly nef if $\int_{X}\alpha\wedge\omega^{n-p}>0$.
\end{defn}
For example, if $[\mu]\in H^{1,1}(X,\mathbb{R})$ is nef, then $[\mu^{p}]$ as a $(p,p)$-cohomology is nef in also nef.   
When $p=1$, the definition of nefness above is equivalent to the usual definition (see \cite{DP}), but the definition of null locus differs. We are not sure whether the definition of null locus for general $p$ is appropriate, but it is sufficient for our applications. Another interesting question is whether, for general $p$, the null locus is an analytic subvariety. Very recently, Fu-Wang-Wu \cite{FWW} published an article that also provided a definition of weak 
$k$-K\"ahler hyperbolicity and discussed their birational invariance. Now, we provide a sufficient condition for a form to be nef.
\begin{prop}\label{p:1}
If there exists a K\"ahler form such that for any $\epsilon>0$ we have a $(p,p)$-form $\alpha_{\epsilon}\in[\alpha]$ such that $\alpha_{\epsilon}+\epsilon\omega^{p}>0$ then $[\alpha]$ is nef.
\end{prop}
\begin{proof}
For any 
$m$-dimensional irreducible analytic subvarieties $V\subset X$, we have
\begin{equation*}
\int_{V}(\alpha_{\epsilon}+\epsilon\omega^{p})\wedge\omega^{m-p}>0.
\end{equation*}
Let $\epsilon\to 0$, then we get $\int_{V}\alpha\wedge\omega^{m-p}\geq 0$.
\end{proof}

\begin{defn}
A K\"ahler manifold $(X,\omega)$ is called weakly $p$-K\"ahler hyperbolic if there exists a strictly nef $(p,p)$-cohomology $[\alpha]\in H^{p,p}(X,\mathbb{R})$ such that it is $\widetilde{d}(\text{bounded})$. It is called weakly $p$-K\"ahler exact if $\tilde{\alpha}$ is exact. 
\end{defn}

\begin{defn}
Let $(X,\omega)$ be a K\"ahler exact manifold. Let $\mathcal{W}_{X}$ be the positive convex cone of all K\"ahler exact strictly nef $(p,p)$-cohomology classes $[\alpha]\in H^{p,p}(X,\mathbb{R})$.  Define the degeneracy set $Z_{X}$ by
\begin{equation*}
Z_{X}=\cap_{[\alpha]\in\mathcal{W}_{X}}Null([\alpha]).
\end{equation*}
\end{defn}

\begin{prop}
If $(X,\omega)$ is weakly $p$-K\"ahler hyperbolic, then it is weakly $q$-K\"ahler hyperbolic for $\forall p<q$.
\end{prop}
\begin{proof}
Let $[\alpha]\in H^{p,p}(X,\mathbb{R})$ be a strictly nef class which is $\widetilde{d}(\text{bounded})$. By definition, it is easy to verify that 
$\alpha\wedge\omega^{q-p}$ is nef and big. Suppose $\pi^{*}\alpha=\widetilde{d}\beta$, $\beta$ is bounded. Then $\pi^{*}(\alpha\wedge\omega^{q-p})=\widetilde{d}(\beta\wedge\pi^{*}\omega^{q-p})$, and $\beta\wedge\pi^{*}\omega^{q-p}$ is bounded.
\end{proof}

\begin{prop}
Let $X$ be a $p$-K\"ahler hyperbolic manifold. $N$ be any other K\"ahler manifold. Then $X\times N$ is weakly $p$-K\"ahler hyperbolic.
\end{prop}
\begin{proof}
Let $\pi_{1}:\widetilde{X}\rightarrow X$ and $\pi_{2}:\widetilde{N}\rightarrow N$ be the universal covering map. Then we have the following diagram
\begin{equation*}
\begin{tikzcd}
X\times N  \arrow[d, "p_{1}"]
& \widetilde{X}\times \widetilde{N} \arrow[l, "\pi_{1}\times\pi_{2}"] \arrow[d, "\widetilde{p}_{1}"] \\
X 
& \widetilde{X} \arrow[l, "\pi_{1}"]
\end{tikzcd}
\end{equation*}
Suppose $\alpha>0$ is a positive $(p,p)$-form on $X$ such that $\pi_{1}^{*}\alpha=\widetilde{d}\beta$ and $\beta$ is bounded. Then $p_{1}^{*}[\alpha]$ is nef by Proposition \ref{p:1}, one can also check that it is strictly nef. The map
$\pi_{1}\times\pi_{2}:\widetilde{X}\times\widetilde{N}\rightarrow X\times N$ is a universal covering map. Then $(\pi_{1}\times\pi_{2})^{*}p_{1}^{*}\alpha=\widetilde{p}_{1}^{*}\pi_{1}^{*}\alpha=\widetilde{d}(\widetilde{p}_{1}^{*}\beta)$. Let $p_{1}^{*}\omega_{X}+p_{2}^{*}\omega_{N}$ be a K\"ahler form on $X\times N$, then $\widetilde{p}_{1}^{*}\beta$ is bounded with respect to this metric. 
\end{proof}

\begin{prop}
Let $X$ be a $p$-K\"ahler hyperbolic manifold and $\nu:Y\rightarrow X$ is a modification. Then $Y$ is weakly $p$-K\"ahler hyperbolic manifold.
\end{prop}
\begin{proof}
Let $\alpha>0$ be a $(p,p)$-form on $X$ such that $\pi_{1}^{*}\alpha=\widetilde{d}\beta$ and $\beta$ is bounded. Then $\nu^{*}\alpha\geq 0$ is strictly nef. One can also show that it is $\widetilde{d}(\text{bounded})$. So $Y$ is weakly $p$-K\"ahler hyperbolic manifold.
\end{proof}
One of the most fundamental questions is whether weak $p$-K\"ahler hyperbolicity is birationally invariant. Under the setting of \cite{FWW}, they proved that weak $p$-K\"ahler hyperbolicity is birationally invariant. Their proof was quite difficult. In this paper, we do not discuss this issue.

\section{Properties of weakly $p$-K\"ahler hyperbolic manifold}\label{sec:weakly}

\begin{prop}[linear isoperimetric inequality] 
Let $X$ be a weakly $p$-K\"ahler hyperbolic manifold, $g$ any Riemannian metric on $X$, and consider the Riemannian universal cover $\pi:(\widetilde{X},\widetilde{g})\rightarrow (X,g)$. Then there exists a constant $C > 0$ such that for any bounded domain $\Omega\subset \widetilde{X}$ with $C^1$ boundary, we have the following linear isoperimetric inequality
\begin{equation*}
\Vol_{\widetilde{g}}(\Omega)\leq C \Area_{\widetilde{g}}(\partial\Omega)
\end{equation*}
\end{prop}
\begin{proof}
Let $\omega$ be a K\"ahler form and $[\alpha]\in H^{p,p}(X,\mathbb{R})$ be a strictly nef class which is $\widetilde{d}(\text{bounded})$. Suppose $\pi^{*}\alpha=\widetilde{d}(\beta)$, let $\eta=\beta\wedge\pi^{*}\omega^{n-p}$. Then $\widetilde{d}\eta=\pi^{*}\mu\wedge\pi^{*}\omega^{n-p}$ and $\eta$ is bounded.  Given any metric $g$ on $X$, normalize it such that
\begin{equation*}
\int_{X}dv_{g}=\int_{X}\mu\wedge\omega^{n-p}.
\end{equation*}
As $\int_{X}:H^{2n}(X,\mathbb{R})\rightarrow\mathbb{R}$ is an isomorphism. Then there exists a $(2n-1)$-form $\rho$ on $X$ such that $dv_{g}-\mu\wedge\omega^{n-p}=d\rho$. Then $\pi^{*}dv_{g}=\widetilde{d}(\eta+\pi^{*}\rho)$ and
\begin{equation*}
|\eta+\pi^{*}\rho|_{\tilde{g}}\leq C.
\end{equation*}
Therefore, using Stokes’ formula we have
\begin{equation*}
\begin{split}
\Vol_{\tilde{g}}(\Omega)=&\int_{\Omega}\pi^{*}dv_{g}=\int_{\Omega}\widetilde{d}(\eta+\pi^{*}\rho)\\
=&\int_{\partial\Omega}(\eta+\pi^{*}\rho)\leq C\Area_{\tilde{g}}(\partial\Omega).
\end{split}
\end{equation*}
\end{proof}

\begin{cor}
Let $X$ be a weakly $p$-K\"hler hyperbolic manifold, with universal cover $\pi:\widetilde{X}\to X$, and let $g$ be any Riemannian metric on $X$. Then zero is not in the spectrum of the Laplace-Beltrami operator of $(\widetilde{X},\pi^{*}g)$.
\end{cor}
\begin{proof}
See \cite{BDET}.
\end{proof}

\begin{prop}
Let $X$ be a weakly $p$-K\"ahler exact manifold and let $l:Y\rightarrow X$ be an irreducible closed $q$-dimensional complex analytic subvariety, $q\geq p$, such that $Y\nsubseteq Z_{X}$. Then the image
$Im(l_{*}:\pi_{1}(Y)\rightarrow\pi_{1}(X))$ is infinite.
\end{prop}
\begin{proof}
If $Y\nsubseteq Z_{X}$, then by definition, there exists a strictly nef  $(p,p)$-cohomology $[\alpha]$ which is $\widetilde{d}$-exact such that  $\int_{Y}\alpha\wedge\omega^{q-p}>0$. If the image
$Im(l_{*}:\pi_{1}(Y)\rightarrow\pi_{1}(X))$ were finite, then passing to a finite cover $\nu:\hat{Y}\to Y$, we would have
that $l\circ\nu:\hat{Y}\rightarrow X$ lifts to the universal cover $\widetilde{f}:\hat{Y}\rightarrow \widetilde{X}$. We would thus obtain a compact
subvariety $\hat{Y}$ of $\widetilde{X}$ such that $\pi|_{\hat{Y}}:\hat{Y}\to Y$ is finite. We would therefore get a contradiction,
since on the one hand this would imply $\int_{\widetilde{Y}}\pi^{*}\alpha\wedge\omega^{q-p}>0$ and on the other hand this integral
has to be zero since $\pi^{*}\alpha\wedge\omega^{q-p}$ is exact.
\end{proof}

\begin{thm}\label{thm:2}
If $(X,\omega)$ is a weakly $p$-K\"ahler exact manifolds. Let $f:\mathbb{CP}^{k}\times\mathbb{CP}^{q-k}\rightarrow X$, $q\geq p$, be a holomorphic map non-degenerate at some point. Then we have $Im(f)\subset Z_{X}$.
\end{thm}
\begin{proof} 
Let $\pi:\widetilde{X}\rightarrow X$ be the universal cover map. Since $\mathbb{CP}^{k}\times\mathbb{CP}^{q-k}$ is simply connected. Then we have a lift map $\widetilde{f}:\mathbb{CP}^{k}\times\mathbb{CP}^{q-k}\rightarrow \widetilde{X}$ satisfies $\pi\circ\widetilde{f}=f$. 
\begin{equation*}
\begin{tikzcd}
& \widetilde{X} \arrow[d, "\pi"] \\
\mathbb{CP}^{k}\times\mathbb{CP}^{q-k} \arrow[ru,"\widetilde{f}"] \arrow[r,"f"]
& X 
\end{tikzcd}
\end{equation*}
Let $[\alpha]\in H^{p,p}(X,\mathbb{R})$ be a strictly nef class and $\pi^{*}\alpha=\widetilde{d}\beta$. Then $f^{*}(\alpha\wedge\omega^{q-p})=\widetilde{f}^{*}\pi^{*}(\alpha\wedge\omega^{q-p})=\widetilde{f}^{*}\widetilde{d}(\beta\wedge\pi^{*}\omega^{q-p})=d\widetilde{f}^{*}(\beta\wedge\pi^{*}\omega^{q-p})$. So $\int_{\mathbb{CP}^{k}\times\mathbb{CP}^{q-k}}f^{*}(\alpha\wedge\omega^{q-p})=0$. Since $f$ is non-degenerate at some point, then $Im(f)$ is a $q$-dimensional subvariety of $X$. Suppsose $Z_{i}$ is its irreducible component and $n_{i}$ is its multiplicity. Then $\sum_{i}n_{i}\int_{Z_{i}}\alpha\wedge\omega^{q-p}=0$. Since $[\mu]$ is arbitrary, it follows that $Im(f)\subset Z_{X}$.
\end{proof}

\begin{cor}
If $q\geq\dim Z_{X}$, then there does not exist a non-degenerate holomorphic map $f:\mathbb{CP}^{k}\times\mathbb{CP}^{q-k}\rightarrow X$. 
\end{cor}

Let $f:\mathbb{C}^{p}\rightarrow X$ be a holomorphic map into a weakly $p$-K\"ahler exact manifold which is non-degenerate at some point. A natural question arises: is it similar to the case of weakly K\"ahler hyperbolic manifolds, where the image of this map falls within the degenerate set $Z_{X}$? In \cite{MP} and  \cite{HM}, the authors proved for $p$-K\"ahler hyperbolic manifold that if certain growth conditions are added, then there does not exist such a holomorphic map. Their proof method seems not to be directly applicable to our case. If we use a stronger definition than weak $p$-K\"ahler hyperbolicity, such as semi-positive $p$-K\"ahler hyperbolicity, it seems that we can apply their methods to reach our conclusion. However, in general cases, we may need an approach similar to the one applied in \cite{BDET} using Ahlfors’ currents.

\section{Vector bundles over weakly K\"ahler hyperbolic manifolds}\label{sec:vb}

Let $(X,\omega)$ be a weakly K\"ahler hyperbolic manifolds. Let $(E,\bar{\partial}_{E})$ be a rank $r$ holomorphic vector bundle over $X$. For any $p$, since $\bar{\partial}_{E}^{2}=0$ then $(\Omega^{p,\bullet}_{X}(E),\bar{\partial}_{E})$ gives a chain complex. Let $H^{p,\bullet}(X,E)$ be the related cohomology group. Let $h^{p,q}=\dim_{\mathbb{C}}H^{p,q}(X,E)$. The $p$-th Euler characteristic $\chi^{p}(X,E)$ is defined by
\begin{equation*}
\chi^{p}(X,E)=\sum_{q=0}^{n}(-1)^{q}h^{p,q}(X,E).
\end{equation*}
For any Hermitian $H$ on $E$. Let $D_{H}$ be the Chern connection with respect to $H$ and $\bar{\partial}_{E}$, $F_{H}$ is its curvature.
Let 
\begin{equation*}
C_{E,H,p}=\sup\{\frac{\langle[\sqrt{-1}F_{\widetilde{H}},\Lambda_{\widetilde{\omega}}]\eta,\eta\rangle_{L^{2}}}{\langle\eta,\eta\rangle_{L^{2}}}|\forall\eta\in L^{2}\Omega^{p,0}(\widetilde{X},\widetilde{E})\},
\end{equation*}
where $\widetilde{\bullet}$ stands for the object in the universal cover. Let
\begin{equation*}
\begin{split}
C_{E,p}=\inf\{C_{E,H,p}|\ \text{$H$ is a Hermitian metric on $E$.}\}
\end{split}
\end{equation*}
be a nonnegative constant that depends only on $E$ and $(X,\omega)$.

First, Let's recall some basic concepts and theorem about $L^{2}$-Hodge theory.

\subsection{$L^{2}$-Hodge theory and $L^{2}$-index theorem}
Let $(X,\omega)$ be a complete K\"ahler manifold. $(E,\bar{\partial}_{E})$ is a holomorphic vector bundle over $X$. Fix a Hermitian metric $H$ on $E$, we can define the $\bar{\partial}_{E}$-Laplace by
\begin{equation*}
\Delta_{\bar{\partial}_{E}}=\bar{\partial}_{E}\bar{\partial}_{E}^{*}+\bar{\partial}_{E}^{*}\bar{\partial}_{E}:C_{0}^{\infty}(X,\Omega^{p,q}\otimes E)\rightarrow C_{0}^{\infty}(X,\Omega^{p,q}\otimes E).
\end{equation*}
where $\bar{\partial}_{E}^{*}$ is the adjoint operator of $\bar{\partial}_{E}$ with respect to the $L^{2}$-inner product. Let $L^{2}\Omega^{p,q}(X,E)$ be the space of $L^{2}$-integrable $(p,q)$ forms valued in $E$, which is of course the completion of $C_{0}^{\infty}(X,\Omega^{p,q}\otimes E)$. The $\bar{\partial}_{E}$-Laplace operator can be extended to a densely defined closed operator on $L^{2}\Omega^{p,q}(X,E)$. Let $\mathcal{D}^{p,q}\subset L^{2}\Omega^{p,q}(X,E)$ be the definition domain of $\Delta_{\bar{\partial}_{E}}$. A $(p,q)$-form $\alpha\in \mathcal{D}^{p,q}$ is called harmonic iff $\Delta_{\bar{\partial}_{E}}\alpha=0$. Let $\mathcal{H}^{p,q}(X,E)$ be the space of $L^{2}$-harmonic $(p,q)$-forms. Using the method of cutoff functions, it can be shown that $\alpha$ harmonic is equivalent to $\bar{\partial}_{E}\alpha =\bar{\partial}_{E}^{*}\alpha=0$ and it admits a Hodge decomposition
\begin{equation*}
L^{2}\Omega^{p,q}(X,E)=\mathcal{H}^{p,q}(X,E)\oplus\overline{Im(\bar{\partial}_{E})}\oplus\overline{Im(\bar{\partial}_{E}^{*})}
\end{equation*}

Let us now recall some definitions and properties of $L^2$-Hodge numbers and $L^2$ Euler characteristic numbers.

Let $\Gamma$ be a discrete faithful group that acts isometrically on $(X,E)$ such that the quotient space $X/\Gamma$ is a compact K\"ahler manifold and $E$ is pull back of a holomorphic vector bundle $V$ over $X/\Gamma$. This action induces a unitary action on $L^{2}\Omega^{p,q}(X,E)$. Moreover, $\mathcal{H}^{p,q}(X,E)$ is a $\Gamma$-invariant subspace. Such a Hilbert space together with a unitary $\Gamma$-action is called a $\Gamma$-module. To each $\Gamma$-module, one assigns the Von Neumann dimension, also called $\Gamma$-dimension, satisfy
\begin{equation*}
0 <\dim_{\Gamma}\mathcal{H}^{p,q}(X,E)\leq\infty,
\end{equation*}
which is a nonnegative real number or $+\infty$. For the precise definition, refer to \cite{BDET,H}, we don't list it here. We mainly use the following properties of $\dim_{\Gamma}$: 
\begin{itemize}
\item[(i)] $\dim_{\Gamma}\mathcal{H}=0\iff \mathcal{H}=\{0\}$.
\item[(ii)] If $\Gamma$ is a finite group, then $\dim_{\Gamma}\mathcal{H}=\dim_{\Gamma}\mathcal{H}/|\Gamma|$.
\item[(iii)] $\dim_{\Gamma}$ is additive. Given a exact sequence 
\begin{equation*}
0\rightarrow\mathcal{H}_{1}\rightarrow\mathcal{H}_{2}\rightarrow\mathcal{H}_{3}\rightarrow 0,
\end{equation*}
one has $\dim_{\Gamma}\mathcal{H}_{2}=\dim_{\Gamma}\mathcal{H}_{1}+\dim_{\Gamma}\mathcal{H}_{3}$.
\end{itemize}
Then the $L^2$-Betti numbers of $(X,E,\Gamma)$ are defined by
\begin{equation*}
h^{p,q}_{\Gamma}(X,E)=\dim_{\Gamma}\mathcal{H}(X,E)<+\infty, \forall 0\leq p,q\leq n
\end{equation*}
The $L^2$-Euler characteristic of $(X,E,\Gamma)$ is
\begin{equation*}
\chi_{\Gamma}^{p}(X,E)=\sum_{q=0}^{n}(-1)^{q}h_{\Gamma}^{p,q}(X,E)
\end{equation*}
By Atiyah's $L^2$-index theorem (see \cite{A}), we know
\begin{equation*}
\chi_{\Gamma}^{p}(X,E)=\chi^{p}(X/\Gamma,V).
\end{equation*}

\subsection{Modification of $X$ and $n$-th Euler characteristic of $E$}
Let $(X,\omega)$ be a weakly K\"ahler hyperbolic manifold of complex dimension $\dim X = n$. Suppose $[\mu]\in H^{1,1}(X,\mathbb{R})$ is nef and big cohomology class which is $\widetilde{d}(\text{bounded})$. We select a K\"ahler current $T\in[\mu]$, which exists since $[\mu]$ is big. By \cite[Proposition 2.3]{Bou}, or \cite[Theorem 3.4 and Lemma 3.5]{DP} there exists a modification 
\begin{equation}\label{eq:md}
\nu:X^{'}\rightarrow X,
\end{equation}
an effective $\mathbb{R}$-divisor $D$, and a K\"ahler form $\omega^{'}$ on $X^{'}$ such that the pull-back $\nu^{*}T$ is cohomologous to $\omega+[D]$, where $[D]$ denotes the current of integration along $D$. Moreover, one has that $D=\nu^{-1}(\nu(D))$, and $\nu|_{X^{'}\setminus D}$ is a biholomorphism onto $X\setminus \nu(D)$.

Let $\mu^{'}=\nu^{*}\mu$. For $t\geq 0$, let
\begin{equation*}
c(t)=\frac{\int_{X^{'}}(\mu^{'}+t\omega^{'})^{n}}{\int_{X^{'}}(\omega^{'})^{n}}.
\end{equation*}
For $t\in[0,1]$, it is bounded with respect to $t$. Now, consider the Monge–Amp\'ere equation
\begin{equation*}
(\mu^{'}+t\omega^{'}+\sqrt{-1}\partial\bar{\partial}\phi_{t})^{n}=c(t)(\omega^{'})^{n}
\end{equation*}
where $(\bullet)^{n}$ denotes the non-pluripolar $n$-th power. Since $[\mu]$ is nef,
\begin{equation*}
\int_{X^{'}}c(t)(\omega^{'})^{n}=\int_{X^{'}}(\mu^{'}+t\omega^{'})^{n}
\end{equation*}
is the volume of the class $[\mu^{'}+t\omega^{'}]$. For $t>0$, it has a unique smooth solution satisfies $\sup_{X^{'}}\phi_{t}=0$ by Yau's theorem. For $t=0$, it has a unique $\mu$-psh solution such that $\sup_{X^{'}}\phi_{0}=0$ by \cite{BEGZ}. In \cite{BDET}, the authors showed that $\phi_{0}$ is smooth on $X^{'}\setminus D$ and $\phi_{t}\rightarrow \phi_{0}$ in the sense of $C^{k,\beta}_{loc}(X^{'}\setminus D)$ for all $k\geq 0, 0<\beta<1$. And $\phi_{t}\rightarrow\phi_{0}$ in $L^{1}(X^{'})$. In particular, $\mu^{'}+\sqrt{-1}\partial\bar{\partial}\phi_{0}$ is a smooth K\"ahler metric on $X^{'}\setminus D$.

Now, we can state our last Theorem.
\begin{thm}\label{thm:2}
Let $(X,\omega)$ be a weakly K\"ahler hyperbolic manifolds. For any modifications $(X^{'},\omega^{'})$ obtained in above way, there is a non-negative constant $\kappa^{'}$. Let $(E,\bar{\partial}_{E})$ be a holomorphic vector bundle over $X$. If there exists a modification $(X^{'},\omega^{'})$ such that $\kappa^{'}\cdot C_{E^{'},p}<1$ for $0\leq p\leq n$, then
\begin{equation*}
\chi^{n}(X,E)>0.
\end{equation*}  
\end{thm}

The original proof idea is due to Gromov \cite{G},  Bei-Diverio-Eyssidieux-Trapani \cite{BDET} modified to apply to the weakly K\"ahler hyperbolic case. We also adopt the method of Bei-Diverio-Eyssidieux-Trapani, and in the following sections, we present the key steps of the proof.

\subsection{Vanishing of $L^{2}$-harmonic forms}

The following proposition plays a crucial role in the proof of the spectral gap theorem.

\begin{prop}\cite{BDET} \label{p:BDET}
Let $(M, \omega)$ be a compact K\"ahler manifold, $(\widetilde{M},\widetilde{\omega})\rightarrow(M,\omega)$ is a K\"ahler universal cover map. Let $(E,H)$ be a Hermitian holomorphic vector bundle over $M$, $(\tilde{E},\tilde{H})$ be the pull back of $(E,H)$. Given any $p,q\in \{0, ..., n\}$, let
\begin{equation*}
\Delta_{\bar{\partial}_{\widetilde{E}}}: L^{2}\Omega^{p,q}(\widetilde{M}, \widetilde{E}) \rightarrow L^{2}\Omega^{p,q}(\widetilde{M}, \widetilde{E})
\end{equation*}
be the closure of $\Delta_{\bar{\partial}_{\widetilde{E}}}: \Omega^{p,q}_{c}(\widetilde{M}, \widetilde{E}) \rightarrow \Omega_{c}^{p,q}(\widetilde{M}, \widetilde{E})$ and let $\{\mathcal{E}(\lambda)\}_{\lambda}$ be its spectral
resolution. For any fixed $\lambda_{0}>0$, there exists $\epsilon_{0}(\lambda_{0})>0$ such that if $0<
\epsilon < \epsilon_{0}(\lambda_{0})$, and $U_{\epsilon}$ is an open set in $M$ with
$\Vol_{\omega}(U_{\epsilon}) < \epsilon$, we have:
\begin{equation*}
\int_{\widetilde{U}_{\epsilon}}|\widetilde{\eta}|_{\widetilde{\omega}}^{2}dv_{\widetilde{\omega}}\leq\int_{\widetilde{M}\setminus\widetilde{U}_{\epsilon}}|\widetilde{\eta}|_{\widetilde{\omega}}^{2}dv_{\widetilde{\omega}}
\end{equation*}
for all $\widetilde{\eta}\in im(\mathcal{E}(\lambda_{0}))$, where $\widetilde{U}_{\epsilon}$ is the preimage of $U_{\epsilon}$ through $\pi:\widetilde{M}\rightarrow M.$
\end{prop}

 Now, we can prove the following spectral gap theorem.
\begin{thm}\label{thm:gap}
Let $\nu:(X^{'},\omega^{'})\rightarrow (X,\omega)$ be a modification related to a $\widetilde{d}(\text{bounded})$ big and nef class. There is non-negtive constant $\kappa^{'}$ corresponding to $(X^{'},\omega^{'})$.

Let $E^{'}=\nu^{*}E$ be the pull back holomorphic vector bundle over $X^{'}$. Assume $H^{'}$ be a Hermitian metric on $E^{'}$ such that $\kappa^{'}C_{E^{'},H^{'},p}<\frac{1}{3}$. Let $\pi: (\widetilde{X^{'}},\widetilde{\omega^{'}})\rightarrow(X^{'}, \omega^{'})$ be the universal covering of $X^{'}$ endowed with the pull-back metric $\widetilde{\omega^{'}}:= \pi^{*}\omega^{'}$.  Let $(\widetilde{E^{'}},\widetilde{H^{'}})$ be the pull back Hermitian holomorphic vector bundle over $\widetilde{X^{'}}$. Then for $0\leq p \leq n-1$ there exists a constant $C$ such that for any $\eta\in \mathcal{D}^{p,0}(\widetilde{X^{'}},\widetilde{E^{'}})$, we have
\begin{equation*}
\begin{split}
&\langle\eta,\Delta_{\partial_{\widetilde{H}^{'}}}\eta\rangle_{L^{2}}\geq C\langle\eta,\eta\rangle_{L^{2}}.\\
&\langle\eta,\Delta_{\bar{\partial}_{\widetilde{E}^{'}}}\eta\rangle_{L^{2}}\geq C\langle\eta,\eta\rangle_{L^{2}}
\end{split}
\end{equation*}
where $\Delta_{\partial_{\widetilde{H}^{'}}}=\partial_{\widetilde{H}^{'}}^{*}\partial_{\widetilde{H}^{'}}+\partial_{\widetilde{H}^{'}}\partial_{\widetilde{H}^{'}}^{*}$.
\end{thm}

\begin{proof} For the sake of simplicity let us replace $X$ with $X^{'}$, $[\mu]$ with $\mu^{'}$, $(E,H)$ with $(E^{'},H^{'})$ and $\omega$ with $\omega^{'}$. Let $\pi:\widetilde{X}\rightarrow X$ be the universal cover map. 

We consider the smooth forms $\mu_{t}:=\mu+\sqrt{-1}\partial\bar{\partial}\phi_{t}$ so that $\mu_{t}+t\omega$ is a K\"ahler form on $X$ for all $t > 0$. Let $\alpha$ be a smooth bounded primitive of $\widetilde{\mu}=\pi^{*}\mu$, so that
\begin{equation*}
\alpha_{t}=\alpha+\sqrt{-1}\widetilde{\partial}\widetilde{\phi}_{t}
\end{equation*}
is a smooth bounded primitive of $\widetilde{\mu}_{t}=\pi^{*}\mu$, where $\widetilde{\phi}_{t}=\phi_{t}\circ\pi$.
Now, let $\eta\in C_{0}^{\infty}(\widetilde{X},\Omega^{p,0}\otimes \widetilde{E})$, $\eta^{*\widetilde{H}}\in C^{\infty}_{0}(\widetilde{X},\Omega^{0,p}\otimes \widetilde{E}^{*})$ satisfies
\begin{equation*}
\eta^{*\widetilde{H}}(e)=\widetilde{H}(e,\eta),\quad\forall e\in \widetilde{E}.
\end{equation*}
So
\begin{equation*}
\begin{split}
D_{\widetilde{H}}\eta^{*\widetilde{H}}(e)=&d(\eta^{*\widetilde{H}}(e))-(-1)^{p}\eta^{*\widetilde{H}}(D_{\widetilde{H}}e)\\
=&\widetilde{H}(D_{\widetilde{H}}e,\eta)+\widetilde{H}(e,D_{\widetilde{H}}\eta)-\widetilde{H}(D_{\widetilde{H}}e,\eta)\\
=&\widetilde{H}(e,D_{\widetilde{H}}\eta),
\end{split}
\end{equation*}
and $|D_{\widetilde{H}}\eta|^{2}_{\widetilde{H},\widetilde{\omega}}=|D_{\widetilde{H}}\eta^{*\widetilde{H}}|_{\widetilde{H}^{*},\widetilde{\omega}}^{2}$. For the convenience of writing, we may omit upper and lower indices. Let $k$ be an integer with $0\leq k\leq n-p-1$. Then we have
\begin{equation*}
\begin{split}
&d(\tr(\eta\wedge\eta^{*})\wedge\widetilde{\alpha}_{t}\wedge\widetilde{\mu}^{n-p-k-1}\wedge\widetilde{\omega}^{k})\\
=&d\tr(\eta\wedge\eta^{*})\wedge\widetilde{\alpha}_{t}\wedge\widetilde{\mu}^{n-p-k-1}\wedge\widetilde{\omega}^{k}+\tr(\eta\wedge\eta^{*})\wedge\widetilde{\mu}^{n-p-k}\wedge\widetilde{\omega}^{k}\\
=&\tr(D_{\widetilde{H}}\eta\wedge\eta^{*})\wedge\widetilde{\alpha}_{t}\wedge\widetilde{\mu}^{n-p-k-1}\wedge\widetilde{\omega}^{k}\\
&+(-1)^{p}\tr(\eta\wedge D_{\widetilde{H}}\eta^{*})\wedge\widetilde{\alpha}_{t}\wedge\widetilde{\mu}^{n-p-k-1}\wedge\widetilde{\omega}^{k}
\\
&+\tr(\eta\wedge\eta^{*})\wedge\widetilde{\mu}^{n-p-k}\wedge\widetilde{\omega}^{k}
\end{split}
\end{equation*}
and
\begin{equation*}
\begin{split}
\Big|\int_{\widetilde{X}}(\sqrt{-1})^{p^{2}}\tr(\eta\wedge\eta^{*})\wedge\widetilde{\mu}^{n-p-k}\wedge\widetilde{\omega}^{k}\Big|\leq &C(n,p,k)\|\alpha_{t}\|_{L^{\infty}}\|\widetilde{\mu}^{n-p-k}\|_{L^{\infty}}\int_{\tilde{X}}|D_{\widetilde{H}}\eta||\eta|dv_{\widetilde{\omega}}\\
\leq&C(n,p,k)\|\alpha_{t}\|_{L^{\infty}}\|\widetilde{\mu}^{n-p-k}\|_{L^{\infty}}\|\eta\|_{L^{2}(\widetilde{X},\widetilde{\omega})}\langle \Delta_{D_{\widetilde{H}}}\eta,\eta\rangle_{L^{2}(\widetilde{X},\widetilde{\omega})}^{1/2}
\end{split}
\end{equation*}
where $\Delta_{D_{\widetilde{H}}}=D_{\widetilde{H}}^{*}D_{\widetilde{H}}+D_{\widetilde{H}}D_{\widetilde{H}}^{*}$ and $D_{\widetilde{H}}^{*}=\sqrt{-1}[\Lambda_{\widetilde{\omega}},\bar{\partial}_{\widetilde{E}}-\partial_{\widetilde{H}}]$. Let 
\begin{equation*}
C(t)=C(n,p,k)\|\alpha_{t}\|_{L^{\infty}}\|\widetilde{\mu}^{n-p-k}\|_{L^{\infty}}
\end{equation*}
Since $C^{\infty}_{0}(\widetilde{X},\Omega^{p,0}\otimes\widetilde{E})$ is dense in $L^{2}\Omega^{p,0}(\widetilde{X},\widetilde{E})$, we can conclude that for each $\eta\in\mathcal{D}^{p,0}$ it holds
\begin{equation}\label{eq:pl}
\begin{split}
\Big|\int_{\widetilde{X}}(\sqrt{-1})^{p^{2}}\tr(\eta\wedge\eta^{*})\wedge\widetilde{\mu}^{n-p-k}\wedge\widetilde{\omega}^{k}\Big|
\leq&C(t)\|\eta\|_{L^{2}(\widetilde{X},\widetilde{\omega})}\langle \Delta_{D_{\widetilde{H}}}\eta,\eta\rangle_{L^{2}(\widetilde{X},\widetilde{\omega})}^{1/2}
\end{split}
\end{equation}
Fix $\lambda_{0}=C_{E}+1$, $0<\epsilon<\epsilon_{0}(\lambda_{0})$. Let $U_{\epsilon}$ be a open neighborhood of divisor $D$ such that $\Vol_{\omega}(U_{\epsilon})<\epsilon$. Since $\phi_{t}$ converges to $\phi_{0}$ in $C^{2,\alpha}(X\setminus D)$, and $\mu+\sqrt{-1}\partial\bar{\partial}\phi_{0}>0$ on $X\setminus D$, there
exists a constant $C_{1}=C_{1}(\epsilon)>0$ independent of $t$ such that $\widetilde{\mu}_{t}+t\widetilde{\omega}>C_{1}\widetilde{\omega}$ on $\widetilde{X}\setminus\widetilde{U}_{\epsilon}$ where $\widetilde{U}_{\epsilon}=\pi^{-1}(U_{\epsilon})$. Then by Proposition \ref{p:BDET}, we have 
\begin{equation*}
\int_{\widetilde{X}}|\eta|_{\widetilde{\omega}}^{2}dv_{\widetilde{\omega}}\leq2\int_{\widetilde{X}\setminus\widetilde{U}_{\epsilon}}|\eta|_{\widetilde{\omega}}^{2}dv_{\widetilde{\omega}}
\end{equation*}
for any $\eta\in Im(\mathcal{E}(\lambda_{0}))$. Therefore, for any $\eta\in Im(\mathcal{E}(\lambda_{0}))$, we get
\begin{equation*}
\begin{split}
\int_{\widetilde{X}}|\eta|_{\widetilde{\omega}}^{2}dv_{\widetilde{\omega}}\leq&2\int_{\widetilde{X}\setminus\widetilde{U}_{\epsilon}}|\eta|_{\widetilde{\omega}}^{2}dv_{\widetilde{\omega}}\\
=&2C(n,p)\int_{\widetilde{X}\setminus\widetilde{U}_{\epsilon}}(\sqrt{-1})^{p^{2}}\tr(\eta\wedge\eta^{*})\wedge\widetilde{\omega}^{n-p}\\
\leq&2\frac{C(n,p)}{C_{1}^{n-p}}\int_{\widetilde{X}\setminus\widetilde{U}_{\epsilon}}(\sqrt{-1})^{p^{2}}\tr(\eta\wedge\eta^{*})\wedge(\widetilde{\mu}_{t}+t\widetilde{\omega})^{n-p}\\
\leq&2\frac{C(n,p)}{C_{1}^{n-p}}\int_{\widetilde{X}}(\sqrt{-1})^{p^{2}}\tr(\eta\wedge\eta^{*})\wedge(\widetilde{\mu}_{t}+t\widetilde{\omega})^{n-p}
\end{split}
\end{equation*}
Then by (\ref{eq:pl}), we have
\begin{equation*}
\begin{split}
&\frac{C(n,p)}{C_{1}^{n-p}}\int_{\widetilde{X}}(\sqrt{-1})^{p^{2}}\tr(\eta\wedge\eta^{*})\wedge(\widetilde{\mu}_{t}+t\widetilde{\omega})^{n-p}\\
\leq & \frac{C(n,p)}{C_{1}^{n-p}}\sum_{k=1}^{n-p}C(n,p,k)t^{k}|\int_{\widetilde{X}}(\sqrt{-1})^{p^{2}}\tr(\eta\wedge\eta^{*})\wedge\widetilde{\mu}_{t}^{n-p-k}\wedge\widetilde{\omega}^{k}|\\
\leq& \frac{C(n,p)}{C_{1}^{n-p}}\sum_{k=1}^{n-p-1}C(n,p,k)t^{k}C(t)\|\eta\|_{L^{2}}\langle\Delta_{D_{\widetilde{H}}}\eta,\eta\rangle_{L^{2}}^{1/2}\\
&+\frac{1}{C_{1}^{n-p}}t^{n-p}\int_{\widetilde{X}}|\eta|^{2}_{\widetilde{\omega}}dv_{\widetilde{\omega}}
\end{split}
\end{equation*}
That is
\begin{equation*}
\begin{split}
(1-\frac{2}{C_{1}^{n-p}}t^{n-p})\|\eta\|_{L^{2}}\leq 2\frac{C(n,p)}{C_{1}^{n-p}}\sum_{k=1}^{n-p-1}C(n,p,k)t^{k}C(t)\langle\Delta_{D_{\widetilde{H}}}\eta,\eta\rangle_{L^{2}}^{1/2}.
\end{split}
\end{equation*}
For $t\in(0,\frac{1}{2^{1/(n-p)}}C_{1})$, let
\begin{equation*}
H(t)=\frac{2\frac{C(n,p)}{C_{1}^{n-p}}\sum_{k=1}^{n-p-1}C(n,p,k)t^{k}C(t)}{1-\frac{2}{C_{1}^{n-p}}t^{n-p}}>0.
\end{equation*}
So
\begin{equation}\label{eq:main}
\|\eta\|_{L^{2}}^{2}\leq H^{2}(t)\langle\Delta_{D_{\widetilde{H}}}\eta,\eta\rangle_{L^{2}}.
\end{equation}
Since we have
\begin{equation*}
\Delta_{D_{\widetilde{H}}}=2\Delta_{\bar{\partial_{\widetilde{E}}}}-\sqrt{-1}[F_{\widetilde{H}},\Lambda_{\widetilde{\omega}}].
\end{equation*}
Then
\begin{equation*}
\begin{split}
\langle\Delta_{D_{\widetilde{H}}}\eta,\eta\rangle_{L^{2}}=&2\langle\Delta_{\bar{\partial}_{\widetilde{E}}}\eta,\eta\rangle_{L^{2}}-\langle\sqrt{-1}\Lambda_{\widetilde{\omega}}(F_{\widetilde{H}}\wedge\eta),\eta \rangle_{L^{2}}\\ 
\leq& 2\langle\Delta_{\bar{\partial}_{\widetilde{E}}}\eta,\eta\rangle_{L^{2}}+C_{E,H,p}\|\eta\|_{L^{2}}
\end{split}
\end{equation*}
Combine with (\ref{eq:main}), we get
\begin{equation*}
\begin{split}
(1-H^{2}(t)C_{E,H,p})\|\eta\|_{L^{2}}^{2}\leq 2H^{2}(t)\langle\Delta_{\bar{\partial}_{\widetilde{E}}}\eta,\eta\rangle_{L^{2}}
\end{split}
\end{equation*}
Let $\kappa^{'}=\inf\{H^{2}(t)|\ \forall t\in(0,\frac{1}{2^{1/(n-p)}}C_{1}]\}$. If $\kappa^{'}C_{E}<\frac{1}{3}$, then we can choose $t_{0}\in(0,\frac{1}{2^{1/(n-p)}}C_{1}]$ and Hermitian metric $H$ such that $\frac{1}{3}-H^{2}(t_{0})C_{E,H,p}>0$. Then we have
\begin{equation*}
\begin{split}
\frac{1-H^{2}(t_{0})C_{E,H,p}}{2H^{2}(t_{0})}\|\eta\|_{L^{2}}^{2}\leq \langle\Delta_{\bar{\partial}_{\widetilde{E}}}\eta,\eta\rangle_{L^{2}}
\end{split}
\end{equation*}
For any $\psi\in\mathcal{D}^{p,0}$. Let $\eta=\mathcal{E}(\lambda_{0})(\psi)$ and $v=\psi-\eta$. Then 
\begin{equation*}
\begin{split}
\langle\eta,v\rangle_{L^{2}}=\langle\Delta_{\bar{\partial}_{\widetilde{E}}}\eta,v\rangle_{L^{2}}=0.
\end{split}
\end{equation*}
and
\begin{equation*}
\|v\|_{L^{2}}^{2}\leq\lambda_{0} \langle\Delta_{\bar{\partial}_{\widetilde{E}}}v,v\rangle_{L^{2}}.
\end{equation*}
Therefore we have
\begin{equation*}
\begin{split}
\langle\Delta_{\bar{\partial}_{\widetilde{E}}}\psi,\psi\rangle_{L^{2}}=&\langle\Delta_{\bar{\partial}_{\widetilde{E}}}\eta,\eta\rangle_{L^{2}}+\langle\Delta_{\bar{\partial}_{\widetilde{E}}}v,v\rangle_{L^{2}}\\
\geq&\frac{1-H^{2}(t_{0})C_{E,H,p}}{2H^{2}(t_{0})}\|\eta\|_{L^{2}}^{2}+\lambda_{0}\|v\|_{L^{2}}^{2}\\
\geq& C\|\psi\|_{L^{2}}^{2}
\end{split}
\end{equation*}
where $C_{2}=\min\{1+C_{E},\frac{1-H^{2}(t_{0})C_{E,H,p}}{2H^{2}(t_{0})}\}$. Since $\frac{1}{3}-H^{2}(t_{0})C_{E,H,p}>0$, we have
\begin{equation*}
\begin{split}
\langle\Delta_{\partial_{\widetilde{H}}}\psi,\psi\rangle_{L^{2}}\geq& C\|\psi\|_{L^{2}}^{2}
\end{split}
\end{equation*}
where $C=C_{2}-C_{E,H,p}>0$.
\end{proof}

\begin{proof}[Outline of the proof of Theorem \ref{thm:2}]
Let $*:\Omega^{p,0}(\widetilde{X},\widetilde{E})\rightarrow \Omega^{n,n-p}(\widetilde{X},\widetilde{E})$ be the Hodge star operator. The following hold
\begin{equation*}
\begin{split}
*\circ\Delta_{\partial_{\widetilde{H}}}=&\Delta_{\bar{\partial}_{\widetilde{E}}}\circ*\\
\Delta_{\partial_{\widetilde{H}}}\circ*=&*\circ\Delta_{\bar{\partial}_{\widetilde{E}}}.
\end{split}
\end{equation*}
So $\eta\in L^{2}\Omega^{n,n-p}$ is $\Delta_{\bar{\partial}_{\widetilde{E}}}$-harmonic if and only if $*\eta$ is $\Delta_{\partial_{\widetilde{H}}}$-harmonic. Then by Theorem \ref{thm:gap}, we have $\mathcal{H}^{n,n-p}(\widetilde{X}^{'},\widetilde{E}^{'})=\{0\}$ for $0\leq p\leq n-1$. By birational invariance \cite[Corollary 11.4]{Ko} of $L^2$-Hodge numbers in bidegree $(n,q)$, we get $\mathcal{H}^{n,n-p}(\widetilde{X},\widetilde{E})=\{0\}$ for $0\leq p\leq n-1$.

To finish the proof, one also need to prove $\mathcal{H}^{n,0}(\widetilde{X},\widetilde{E})\neq\{0\}$. As shown in \cite{BDET}, this can be achieved through the Vafa-Witten trick. We omit this part here. 

After prove the non-vanishing theorem, we then get $\chi^{n}(X,E)>0$.

\end{proof}

\medskip

{\bf  Acknowledgement:} 
The author are partially supported by NSFC (Grant No. 12141104) and Natural Science Foundation of Jiangsu Province, China (Grant No. BK20241434).


\end{document}